\documentclass[12pt,a4paper,fleqn]{amsart}
\usepackage[utf8]{inputenc}
\usepackage[T1]{fontenc}
\usepackage{a4wide}
\usepackage{amssymb,amsthm}
\usepackage{newtxtext,newtxmath}
\usepackage[mathscr]{euscript}
\usepackage{url}
\usepackage[colorlinks=true,urlcolor=blue,linkcolor=blue,citecolor=blue]{hyperref}
\hypersetup{pdftitle={Uniform Property (S)},pdfauthor={William B. Johnson and Tomasz Kania}}

\newtheorem{theorem}{Theorem}[section]
\newtheorem{lemma}[theorem]{Lemma}
\newtheorem{corollary}[theorem]{Corollary}
\newtheorem{proposition}[theorem]{Proposition}

\theoremstyle{definition}
\newtheorem{definition}[theorem]{Definition}
\newtheorem{example}[theorem]{Example}
\newtheorem{remark}[theorem]{Remark}
\theoremstyle{plain}

\newcounter{maintheorem}

\newtheorem{mainth}[maintheorem]{Theorem}

\newtheorem*{theorembprime}{Theorem B$^{\prime}$}
\renewcommand{\le}{\leqslant}
\renewcommand{\ge}{\geqslant}

\newcommand{\vertiii}[1]{{\left\vert\kern-0.25ex\left\vert\kern-0.25ex\left\vert #1 
    \right\vert\kern-0.25ex\right\vert\kern-0.25ex\right\vert}}
\newcommand{\N}{\mathbb{N}}

\newcounter{smallromans}

\newenvironment{romanenumerate}
{\begin{list}{{\normalfont\textrm{(\roman{smallromans})}}}%
  {\usecounter{smallromans}\setlength{\itemindent}{0cm}%
   \setlength{\leftmargin}{5.5ex}\setlength{\labelwidth}{5.5ex}%
   \setlength{\topsep}{.5ex}\setlength{\partopsep}{.5ex}%
   \setlength{\itemsep}{0.1ex}}}%
{\end{list}}

\newcounter{smallromansdash}

{\end{list}}

\newcounter{bigromans}

{\end{list}}

\title[Uniform Property (S)]{Uniform Property (S)}
\author[W.~B.~Johnson]{William B.~Johnson}
\address[W.~B.~Johnson]{Department of Mathematics\\Texas A\&M University\\College Station, TX 77843, U.S.A.}
\email{johnson@math.tamu.edu}
\author[T.~Kania]{Tomasz Kania}
\address[T.~Kania]{Mathematical Institute\\Czech Academy of Sciences\\\v Zitn\'a 25\\115 67 Praha 1\\Czech Republic and Institute of Mathematics and Computer Science\\Jagiellonian University\\{\L}ojasiewicza 6, 30-348 Krak\'{o}w, Poland}
\email{kania@math.cas.cz, tomasz.marcin.kania@gmail.com}
\thanks{RVO: 67985840.}
\date{}

\begin{document}
\begin{abstract}
We introduce and investigate a quantitative version of Steinhaus' property~\((S)\) for Banach spaces, called the \emph{uniform property~\((S)\)}.  A Banach space~\(X\) is said to have uniform~\((S)\) if for every pair of distinct unit vectors \(x,y\in X\) and every~\(a>0\), the difference of the perturbed norms
\[
   \sup_{\|z\|\le a}\big|\|x+z\|-\|y+z\|\big|
\]
is bounded below by a positive function of~\(a\) and~\(\|x-y\|\).  We compute this modulus exactly for the real spaces \(L_1(\mu;\mathbb R)\) with atomless measure~\(\mu\),
\[
   U_{L_1(\mu)}(d;a)=\Big(\tfrac{4a}{2+d}\wedge 1\Big)d.
\]
For the corresponding complex-valued spaces the modulus is \(d\min\{a,1\}\). Uniform complex convexity implies uniform property~\((S)\), so the predual of every complex von Neumann algebra has uniform~\((S)\). The class of spaces with uniform~\((S)\) is stable under ultrapowers, Bochner-\(L_1\) constructions, and contains all Gurari\u{\i} spaces as well as Banach lattices of almost universal disposition.  In particular, every non-zero Banach space embeds isometrically into a non-strictly convex Banach space of the same density having uniform~\((S)\).  We further exhibit an explicit equivalent renorming of~\(\ell_1(\Gamma)\),
\[
   \|x\|_S=\big(\|x\|_1^2+\|x\|_2^2\big)^{1/2},
\]
which endows~\(\ell_1(\Gamma)\) and all its ultrapowers with uniform~\((S)\).
These results settle, in~ZFC, several open questions about the quantitative geometry of property~\((S)\) posed by Kochanek and the second-named author.
\end{abstract}

\maketitle

\section{Introduction}

It is a familiar observation from plane geometry that for every positive integer $n$, it is possible to encircle precisely $n$ points with integer coordinates in the Euclidean plane (this seems to be first mentioned by Steinhaus in 1957 as problem 498 in volume 10 of \emph{Matematyka: Czasopismo dla Nauczycieli}, a journal for Mathematics teachers in Poland, and later also as Problem 24 in \cite[p.~17]{Steinhaus:1965}). In May 1967, Steinhaus put in the New Scottish Book (that revived the traditions of the Lvov Scottish Book in Wroc{\l}aw) problem 782 of whether every measurable subset of the plane that has infinite Lebesgue measure can be moved so that the intersection of the result set with the integer lattice is infinite. Worth mentioning is also a result from plane geometry by Fast and \'Swierczkowski \cite{FastSwierczkowski}, which asserts that for any open bounded region $G$ in the Euclidean plane with a~simple closed curve as its boundary and an integer-valued area, there exists an (isometric copy of the) integral lattice sharing precisely $|G|$ points with $G$. \smallskip

The possibility of encircling any number of lattice points, mentioned by Steinhaus, was generalised by Zwole\'nski \cite{Zwolenski:2011} to arbitrary Hilbert spaces as follows: for every infinite subset $A$ in a~Hilbert space $X$ with a finite intersection with each open ball, there exists a dense subset $Y$ of $X$ such that for all $y \in Y$ and $n \in \N$ there is an open ball $B$ centred at $y$ with $|A \cap B| = n$. Kochanek and the second-named author \cite{KaniaKochanek:2017} expanded beyond Hilbert spaces, characterising geometrically Banach spaces for which Zwole\'nski's assertion remains valid; these were termed spaces having \emph{property $(S)$} after Steinhaus. \smallskip

In this paper, we adopt one of the equivalent geometric conditions enumerated in \cite[Theorem A]{KaniaKochanek:2017} as the definition of spaces with property $(S)$.

\begin{definition}
    A normed space $X$ has \emph{property $(S)$} if, for every pair of distinct points $x$ and $y$ in the unit sphere of $X$ and every $a > 0$, there exists a vector $z$ in $X$ with $\|z\| < a$ such that $\|x + z\| \neq \|y + z\|$.
\end{definition}

This definition includes possibly incomplete spaces; however, completeness is necessary to establish the equivalence between the various properties enumerated in \cite[Theorem A]{KaniaKochanek:2017}. (In \cite{KaniaKochanek:2017}, the above-introduced property was called {property} $(S^{\prime\prime})$, while a metric space $X$ was said to have property $(S)$ if, for any infinite subset $A$ of $X$ with a finite intersection with every ball, there exists a dense subset $Y$ of $X$ such that for all $y \in Y$ and $n \in \N$ there is a ball $B$ centred at $y$ with $|A \cap B| = n$.) The goal of this paper is to quantify property $(S)$, specifically to introduce \emph{uniform} property $(S)$, and to demonstrate that various well-known Banach spaces possess uniform property $(S)$, or at least can be equivalently renormed to have uniform property $(S)$. In the process of doing this, we solve a problem about property $(S)$ that was raised in \cite{KaniaKochanek:2017}.\smallskip

Let $X$ be a normed space. For $x,y\in X$ and $a>0$, we define the $(S)$-\emph{modulus} by
\begin{equation}\label{U(x,y)}
    U_X(x,y;a) = U(x,y;a) :=  \sup_{\|z\| \leqslant a}\big|  \|x+z\| - \|y+z\| \big|.
\end{equation}
For $0<d<2$ and $a>0$ we define
\begin{equation}\label{U(b:a)}
    U_X(d;a) = U(d;a) := \inf_{\overset{x, y \in S_X,}{\|x-y\| \geqslant d}} U(x,y;a).
\end{equation}
\begin{definition}
    A Banach space $X$ has {\sl uniform property (S)} provided that for all $0<d<2, a>0$ one has
    \[
        U_X(d;a) > 0.
    \]
\end{definition}

It is evident that uniform property $(S)$ implies property $(S)$, and a standard compactness argument shows that if $X$ is finite-dimensional and has property $(S)$, then it also has uniform property $(S)$. Uniform property~$(S)$ is preserved under completion and ultrapowers.
The precise comparison for ultrapowers is given in
Lemma~\ref{lem:ultrapower-modulus}. For fixed $a$, the pairwise modulus
satisfies
\[
 |U_X(x,y;a)-U_X(x',y';a)|
 \leqslant\|x-x'\|+\|y-y'\|.
\]
Together with density of the original space in its completion, this also
gives the corresponding assertion about completion, allowing an
arbitrarily small loss in the separation parameter. \smallskip

The modulus of uniform property (S), as defined, may be interpreted as the $p = 1$ case of the whole scale of moduli whose strict positivity is equivalent to uniform property (S). Indeed, for a~normed space $X$ and $p\in [1,\infty)$ we set
\begin{equation}\label{Up(b:a)}
    U_X^{(p)}(d;a) = U^{(p)}(d;a) := \inf_{\substack{x, y \in S_X\\ \|x-y\| \geqslant d}} U^{(p)}(x,y;a)\quad (d\in (0,2),\ a>0),
\end{equation}
where
\begin{equation}\label{Up(x,y)}
    U_X^{(p)}(x,y;a) = U^{(p)}(x,y;a) :=  \sup_{\|z\| \leqslant a}\big|  \|x+z\|^p - \|y+z\|^p \big|\quad (x, y\in X,\ a > 0).
\end{equation}
Certainly, $U_X^{(1)}(d;a) = U_X(d;a)$ and for $p > 1$ we have $U_X^{(p)}(d;a) > 0$ if and only if $U_X(d;a) > 0$, which follows (for small enough $a$) from the elementary fact:

\begin{lemma}\label{lem:p-vs-1}
Fix $p>1$ and $a\in(0,1)$. If $r,s\in[1-a,1+a]$, then by the mean value theorem
\[
  p(1-a)^{p-1}\,|r-s|\ \le\ |r^p-s^p|\ \le\ p(1+a)^{p-1}\,|r-s|.
\]
Consequently, for every normed space $X$ and every $a\in(0,1)$ we have, for $0<d<2$,
\[
  p(1-a)^{p-1}\,U_X(d;a)\ \le\ U_X^{(p)}(d;a)\ \le\ p(1+a)^{p-1}\,U_X(d;a).
\]
In particular, for any $a\in(0,1)$ and $p>1$, $U_X^{(p)}(d;a)>0$ if and only if $U_X(d;a)>0$.
\end{lemma}

\begin{proof}
Apply the one-variable mean value theorem to $t\mapsto t^p$ on $[1-a,1+a]$ and note that
the intermediate value $\xi$ lies in $[1-a,1+a]$.
\end{proof}

Let us illustrate it briefly with inner-product spaces.
\begin{proposition}
    Let $X$ be an inner-product space of dimension at least $2$. Then $U_X^{(2)}(d; a) = 2ad$ for $a > 0$ and $d\in (0, 2)$.
\end{proposition}
\begin{proof}
Let \(x, y\in S_X\) and \(z\) with \(\|z\| \leqslant a\). We have
\(\|x+z\|^2 = \|x\|^2 + \|z\|^2 + 2\langle x, z \rangle\) and \(\|y+z\|^2 = \|y\|^2 + \|z\|^2 + 2\langle y, z \rangle\). Since \(x\) and \(y\) are unit vectors, we have
\[
    \big|\|x+z\|^2 - \|y+z\|^2\big|
    = 2\big|\langle x-y, z \rangle\big|.
\]
Maximising over all \(z\) with \(\|z\| \leqslant a\), we get
\[
    \sup_{\|z\| \leqslant a} \big| \langle x-y, z \rangle \big|
    = a \|x-y\|,
\]
hence
\[
\sup_{\|z\| \leqslant a}\big| \|x+z\|^2 - \|y+z\|^2 \big| = 2a \|x-y\|.
\]
Therefore,
\[
U_X^{(2)}(d;a)
=\inf_{\substack{x, y \in S_X\\ \|x-y\| \geqslant d}} 2a\|x-y\|
=2ad,
\]
where the last equality uses that in an inner-product space with $\dim X\ge 2$, for every $d\in(0,2)$ there exist $x,y\in S_X$ with $\|x-y\|=d$.
\end{proof}

It is possible that every Banach space, or at least every separable Banach space, has an equivalent renorming with uniform property $(S)$. Conversely, there might be non-separable Banach spaces that do not have renormings with property $(S)$. In Section \ref{problems}, we discuss these open problems further. \smallskip

\begin{remark}\label{rem:uniform-convexity-bound}
Let $\delta_X$ denote the modulus of convexity of a uniformly convex
Banach space $X$. If $0<d<2$ and $0<a\leqslant d/2$, then
\[
 U_X(d;a)\geqslant
 4a\inf_{d\leqslant s\leqslant2}\frac{\delta_X(s)}s.
\]
In particular, if $\delta_X(s)\geqslant c s^p$ for $0<s\leqslant2$,
where $c>0$ and $p\geqslant1$, then
$U_X(d;a)\geqslant4ca d^{p-1}$ in this range.

Indeed, take $x,y\in S_X$, put $s=\|x-y\|\geqslant d$ and
$m=(x+y)/2$, and set $z=(a/s)(y-x)$. Then $\|z\|=a$ and
\[
 x+z=\left(1-\frac{2a}s\right)x+\frac{2a}s m,
 \qquad
 y=\frac{2a}{s+2a}m+\frac{s}{s+2a}(y+z).
\]
These identities, together with $\|m\|\leqslant1-\delta_X(s)$,
give
\[
 \|x+z\|\leqslant1-\frac{2a}s\delta_X(s),
 \qquad
 \|y+z\|\geqslant1+\frac{2a}s\delta_X(s).
\]
Consequently $U_X(x,y;a)\geqslant4a\delta_X(s)/s$.
Taking the infimum over pairs proves the assertion.
\end{remark}

Property $(S)$ is weaker than strict convexity. Indeed,
\cite[Theorem A and Proposition 1]{KaniaKochanek:2017} show that strictly
convex spaces have property~$(S)$, and that in dimensions at most two the
two properties coincide. Strict convexity alone does not imply uniform
property~$(S)$, as the following example shows.

\begin{example}\label{ex:c0-compact-renorming}
Let $T:c_0\to\ell_2$ be compact and injective; for example,
$(Tx)_n=2^{-n}x_n$. Then
\[
 N(x)=\|x\|_\infty+\|Tx\|_2\qquad(x\in c_0)
\]
is an equivalent strictly convex norm without uniform property~$(S)$.
To see the latter assertion, put
\[
 v_n=e_1+\tfrac12e_n,\qquad c_n=N(v_n)^{-1},\qquad x_n=c_nv_n
 \quad(n\geqslant2).
\]
Compactness gives $Te_n\to0$, so $c_n\to c=(1+\|Te_1\|_2)^{-1}>0$.
For large distinct $n,m$, we have $N(x_n-x_m)\geqslant c/4$.
Fix $a=c/8$. If $N(z)\leqslant a$, the first coordinate determines
both supremum norms for all sufficiently large $n,m$, uniformly in $z$.
Thus
\[
 \big|N(x_n+z)-N(x_m+z)\big|
 \leqslant |c_n-c_m|+\|T(x_n-x_m)\|_2\longrightarrow0.
\]
It follows that $U_{(c_0,N)}(c/4;a)=0$.
\end{example}

Whether $c_0$ admits some equivalent norm with uniform property~$(S)$
remains open.

We begin by listing three classes of Banach spaces with uniform property $(S)$.

\begin{mainth}\label{ThA} The following spaces have uniform property $(S)$:
\begin{romanenumerate}
    \item\label{ThAi}  the non-zero Bochner spaces $L_p(\mu,X)$ for any atomless measure space $(\Sigma, \mathcal{A}, \mu)$ and $p\in [1,\infty)$. In particular,
        \begin{itemize}
            \item $L_1([0,1], C[0,1])$ is a separable space with uniform property $(S)$ universal for all separable Banach spaces;
            \item $L_1([0,1], C(\{0,1\}^{\mathbb N}, L_1[0,1]))$ is a separable Banach lattice with uniform property $(S)$ universal for all separable Banach lattices with respect to isometric lattice embeddings.
        \end{itemize}
    \item\label{ThAii} Gurari\u{\i} spaces of arbitrary density satisfy
        \[ 
            U_{G}(d;a) = \Big( \frac{4a}{2+d} \wedge 1 \Big) d \qquad (0<d \leqslant 2,\; 0<a).
        \]
    \item\label{ThAiii} Banach lattices of almost universal disposition for finitely generated Banach lattices including the one constructed by Tursi \cite{Tursi:2020}.
\end{romanenumerate}
\end{mainth}

The latter part of Theorem~\ref{ThA} is an immediate consequence of clause \eqref{ThAi}, as for any space $X$, the Bochner space $L_1([0,1], X)$ possesses uniform property $(S)$, and $X$ embeds therein via almost everywhere constant functions (which constitutes a lattice embedding if $X$ is a Banach lattice). It is clear that $L_1([0,1], X)$ has the same density as $X$. The lattice universality asserted in clause \eqref{ThAi} follows from \cite[Theorem~1.1]{LeungLiOikhbergTursi} and the embedding by constant functions. \smallskip

Gurari\u{\i} spaces are Banach spaces of almost universal disposition
for finite-dimensional normed spaces. Tursi~\cite{Tursi:2020} constructed
a separable Banach lattice of almost universal disposition for finitely
generated Banach lattices. The following consequence of
clause~\eqref{ThAi} resolves in ZFC a question from
\cite{KaniaKochanek:2017}.

\begin{corollary}\label{gurarii}
Every non-zero Banach space embeds isometrically into a non-strictly convex
space of the same density with uniform property~$(S)$. For Banach lattices,
the embedding may be chosen to be lattice-isometric.
\end{corollary}
\begin{proof}
Embed $X$ into $L_1([0,1],X)$ as the constant functions. The embedding is
lattice-isometric when $X$ is a Banach lattice, and the ambient space has
the same density as $X$ and uniform property~$(S)$ by Theorem~\ref{ThA}.
For $u\in S_X$, the functions $2u\mathbf1_{[0,1/2]}$ and
$2u\mathbf1_{(1/2,1]}$ are distinct unit vectors whose midpoint is a unit
vector. Hence the ambient space is not strictly convex.
\end{proof}

Pe{\l}czy\'nski's universal basis space $P$ \cite{Pelczynski:1971} has a~basis with the property that each space with a~basis embeds as a complemented subspace thereof. In particular, $P$ is isomorphic to $L_1(P)$. We may thus record the following corollary.\smallskip

\begin{corollary}
    The universal basis space $P$ can be equivalently renormed to have uniform property $(S)$.
\end{corollary}

By applying Theorem~\ref{ThA}, we can positively answer the question left open in \cite{KaniaKochanek:2017} regarding the existence (within ZFC) of spaces with property $(S)$ that do not have equivalent strictly convex norms. Although \cite[Theorem B]{KaniaKochanek:2017} already provided such an example, it was based on the assumption that the continuum is a real-measurable cardinal number. Since not every Banach space admits a strictly convex equivalent renorming (\emph{e.g.}, $\ell_\infty(\Gamma)$ for an uncountable set $\Gamma$  or $\ell_\infty / c_0$ as proved by Day \cite{day} and Bourgain \cite{bourgain}, respectively), the corresponding Bochner spaces possess the desired property since the property of having an equivalent strictly convex norm is preserved in closed subspaces. \smallskip

The proof of Theorem~\ref{ThA} yields a suboptimal bound for $U_{L_1}(d, a)$ (\emph{i.e.}, when $X$ is one-dimensional). Theorem~\ref{L1} determines the exact value of $U_{L_1}(d; a)$ for all $d \leqslant 2$ and all $a > 0$.

\begin{mainth}\label{L1} 
    Let $\mu$ be an atomless measure with $L_1(\mu;\mathbb R)\ne\{0\}$. This real space has uniform property $(S)$ with
        \begin{equation}\label{UL1}
            U_{L_1(\mu)}(d;a) =  \Big( \frac{4a}{2+d} \wedge 1 \Big) d\qquad (0<d \leqslant 2,\; a > 0).
        \end{equation}
\end{mainth}

The complex scalar field leads to a different modulus. Nevertheless, complex
uniform convexity provides a general sufficient condition for uniform
property~$(S)$, as shown in Proposition~\ref{prop:complex-convexity-S} below.
In particular, we obtain the following consequence.

\begin{theorembprime}\label{thm:complex-predual-uniformS}
Let $A$ be a complex $C^*$-algebra. Then $A^*$ and every closed complex
linear subspace of $A^*$ have uniform property~{\rm(S)}. In particular,
the complex predual $\mathcal M_*$ of every von Neumann algebra
$\mathcal M$, regarded as a real Banach space, has uniform property~{\rm(S)}.
\end{theorembprime}

The following elementary observation will be used to show that uniform property~$(S)$ need
not pass to biduals.

\begin{lemma}\label{lem:L-summand-obstruction}
Let $Z=Y\oplus_1 W$. If $x,y\in S_Y$, then
\[
    U_Z((x,0),(y,0);a)=U_Y(x,y;a)\qquad(a>0).
\]
Consequently, if an $L$-summand of a Banach space fails uniform property~$(S)$, then the
whole space fails uniform property~$(S)$. If the $L$-summand fails property~$(S)$, then so
does the whole space.
\end{lemma}

\begin{proof}
For $z=(u,w)\in Z$ with $\|z\|=\|u\|+\|w\|\le a$, we have
\[
\begin{split}
&\big|\|(x,0)+z\|-\|(y,0)+z\|\big| \\
&\quad=\big|\|x+u\|+\|w\|-\|y+u\|-\|w\|\big|
\\ &\quad=\big|\|x+u\|-\|y+u\|\big|
     \le U_Y(x,y;a).
\end{split}
\]
Taking the supremum over all such $z$ gives
$U_Z((x,0),(y,0);a)\le U_Y(x,y;a)$. The reverse inequality follows by restricting to
vectors of the form $(u,0)$. The last assertions follow at once after representing a
Banach space with an $L$-summand $Y$ as $Y\oplus_1 W$.
\end{proof}

\begin{lemma}\label{lem:p-summand-modulus}
Let $1\leqslant p<\infty$ and $Z=Y\oplus_p W$. If $x,y\in S_Y$, then
\[
 U_Z^{(p)}((x,0),(y,0);a)=U_Y^{(p)}(x,y;a)
 \qquad(a>0).
\]
\end{lemma}
\begin{proof}
For $z=(u,w)$ with $\|z\|\leqslant a$, the $W$-terms cancel:
\[
 \big|\|(x,0)+z\|^p-\|(y,0)+z\|^p\big|
 =\big|\|x+u\|^p-\|y+u\|^p\big|
 \leqslant U_Y^{(p)}(x,y;a).
\]
The reverse inequality follows by taking $w=0$.
\end{proof}

\begin{corollary}\label{cor:L1-bidual-counterexample}
For real Banach spaces, uniform property~$(S)$ is not inherited by
biduals. More precisely, $L_1([0,1];\mathbb R)$ has uniform
property~$(S)$, whereas its bidual fails even property~$(S)$.
\end{corollary}

\begin{proof}
By Theorem~\ref{L1}, the space $L_1[0,1]$ has uniform property~$(S)$. We prove that its
bidual does not.

Let $K$ be a compact Hausdorff space and identify $C(K)^*$ with the Banach space $M(K)$ of
regular signed measures on $K$. Let
\[
    E_K=\overline{\operatorname{span}}\{\delta_t:t\in K\}\subset C(K)^*.
\]
The map
\[
    (a_t)_{t\in K}\in \ell_1(K)\longmapsto \sum_{t\in K}a_t\delta_t\in E_K
\]
is an onto isometry, so $E_K$ is isometric to $\ell_1(K)$.

We claim that $E_K$ is an $L$-summand of $C(K)^*$. Given $\mu\in M(K)$, the set
$A(\mu)=\{t\in K:\mu(\{t\})\ne0\}$ is countable, and
$\sum_{t\in A(\mu)}|\mu(\{t\})|\le \|\mu\|$. Hence
\[
    P\mu:=\sum_{t\in A(\mu)}\mu(\{t\})\delta_t
\]
defines an element of $E_K$. The measure $\mu_c:=\mu-P\mu$ satisfies
$\mu_c(\{t\})=0$ for every $t\in K$ and, in particular, $|\mu_c|(A(\mu))=0$, whereas
$P\mu$ is supported by $A(\mu)$. Thus $P\mu$ and $\mu_c$ are mutually singular. The assignment $P$ is linear, $P^2=P$, and
\[
    \|\mu\|=\|P\mu\|+\|\mu-P\mu\|.
\]
Consequently $P$ is an $L$-projection and $C(K)^*=E_K\oplus_1\ker P$.

If $K$ contains two distinct points $s$ and $t$, then $E_K$ fails property~$(S)$. Indeed,
inside the isometric copy of $\ell_1(K)$ set
\[
    x=\tfrac12\delta_s+\tfrac12\delta_t,
    \qquad
    y=\tfrac14\delta_s+\tfrac34\delta_t .
\]
Then $x,y\in S_{E_K}$ and $\|x-y\|=1/2$. If
$z=\sum_{r\in K}z_r\delta_r\in E_K$ and $\|z\|\le 1/4$, then the coefficients of
$x+z$ and $y+z$ at $s$ and $t$ are still non-negative, and therefore
\[
    \|x+z\|=1+z_s+z_t+\sum_{r\in K\setminus\{s,t\}}|z_r|=
    \|y+z\|.
\]
Thus $U_{E_K}(x,y;1/4)=0$, so $E_K$ fails property~$(S)$ and hence also uniform
property~$(S)$. By Lemma~\ref{lem:L-summand-obstruction}, $C(K)^*$ fails property~$(S)$.

Finally, by the Kakutani--Gelfand representation theorem, $L_\infty[0,1]$ is isometric to
$C(K)$ for a compact Hausdorff space $K$ with at least two points. Hence
\[
    L_1[0,1]^{**}=L_\infty[0,1]^*\cong C(K)^*,
\]
and the preceding paragraph shows that $L_1[0,1]^{**}$ fails property~$(S)$.
\end{proof}

The scalar field is essential here. The bidual of
$L_1([0,1];\mathbb C)$ has uniform property~$(S)$ by Theorem~B$^{\prime}$,
because it is the dual of the complex $C^*$-algebra $L_\infty[0,1]$.

\begin{mainth}\label{ThmC}
    Suppose that $X$ is a Banach space that is not super-reflexive. Then
    \[
        U_X(d;a)\leqslant \Big( \frac{4a}{2+d} \wedge 1 \Big) d\qquad (0<d \leqslant 2,\; a > 0).
    \]
\end{mainth}
Theorems~\ref{L1} and~\ref{ThmC} give the exact modulus of
Gurari\u{\i} spaces in Theorem~\ref{ThA}.

The uniform-convexity estimate in Remark~\ref{rem:uniform-convexity-bound}
has the correct order for $\ell_p$, $p>2$, when
the perturbation radius is at most half the separation. The constants
in the comparison below depend on $p$.

\begin{corollary}\label{cor:ellp-modulus}
Let $p>2$. For $0<d<2$ and $0<a\leqslant d/2$,
\[
 \frac{2^{2-p}}p\,a d^{p-1}
 \leqslant U_{\ell_p}(d;a)
 \leqslant 2^p p\,a d^{p-1}.
\]
\end{corollary}
\begin{proof}
For real scalars $s,t$ and $p\geqslant2$, convexity gives
\[|s+t|^p+|s-t|^p\leqslant2^{p-1}(|s|^p+|t|^p).\]
Indeed, apply $A^{p/2}+B^{p/2}\leqslant(A+B)^{p/2}$ to
$A=|s+t|^2$, $B=|s-t|^2$, then use
$(|s|^2+|t|^2)^{p/2}\leqslant2^{p/2-1}(|s|^p+|t|^p)$.
Summing coordinatewise gives, for $x,y\in S_{\ell_p}$,
\[
 \left\|\frac{x+y}{2}\right\|^p+
 \left\|\frac{x-y}{2}\right\|^p\leqslant1.
\]
It follows that
\[
 \delta_{\ell_p}(s)
 \geqslant1-\left(1-(s/2)^p\right)^{1/p}
 \geqslant\frac{s^p}{p2^p}.
\]
Apply Remark~\ref{rem:uniform-convexity-bound} for the lower estimate.

For the upper estimate, put $\beta=d/2$ and
$\alpha=(1-\beta^p)^{1/p}$, and take
\[
 x=(\alpha,\beta,0,\ldots),\qquad
 y=(\alpha,-\beta,0,\ldots).
\]
These are unit vectors at distance $d$. If $z\in\ell_p$ and $\|z\|\leqslant a$,
as in Lemma~\ref{lem:p-summand-modulus}, all coordinates except the second cancel in the difference of $p$th powers:
\[
 \big|\|x+z\|^p-\|y+z\|^p\big|
 =\big||\beta+z_2|^p-|-\beta+z_2|^p\big|
 \leqslant2pa d^{p-1}.
\]
Indeed, $|z_2|\leqslant a\leqslant\beta$ and the mean value theorem
bounds the last expression by
$2p|z_2|(\beta+|z_2|)^{p-1}\leqslant2pa d^{p-1}$.

Write $r=\|x+z\|$ and $s=\|y+z\|$. The triangle inequality implies
\[
 \max\{r,s\}\geqslant
 \max\{d/2,1-a\}
 \geqslant\max\{d/2,1-d/2\}\geqslant\tfrac12.
\]
For non-negative $r,s$ and $p\geqslant1$,
\[
 |r^p-s^p|\geqslant
 \max\{r,s\}^{p-1}|r-s|.
\]
Thus $|\|x+z\|-\|y+z\||\leqslant2^p pa d^{p-1}$ for every
$z$ in the whole $\ell_p$-ball, proving the upper estimate.
\end{proof}

Finally, we prove that $\ell_1$ (and all its ultrapowers) can be equivalently renormed to have uniform property (S).
\begin{mainth}\label{thm:l1renorming}
Let $\Gamma$ be a set. Then the norm 
\[
    \|x\|_S^2 = \|x\|_{\ell_1(\Gamma)}^2 + \|x\|_{\ell_2(\Gamma)}^2\quad\big(x\in \ell_1(\Gamma)\big)
\]
is an equivalent norm on the real space $\ell_1(\Gamma;\mathbb R)$
with uniform property~{\rm(S)}. More precisely,
\[U_{(\ell_1(\Gamma),\|\cdot\|_S)}(d;a)
\geqslant2^{-56}\min\{a,d,1\}^{9}\qquad(0<d<2,\ a>0).\]
\end{mainth}
The proof combines the two components of this norm. It also leads to the
following consequence for real biduals.

\begin{corollary}\label{cor:l1-bidual-renorming}
For every set $\Gamma$, the real Banach space
$\ell_1(\Gamma;\mathbb R)^{**}$ admits an equivalent norm with uniform
property~{\rm(S)}.
\end{corollary}

Finally, the class of separable spaces with uniform property~$(S)$ is Borel.
More precisely, Proposition~\ref{prop:uniformS-Pi4} gives a $\Pi^0_4$ upper
bound in the Wijsman quotient coding.

\section{Preliminaries}
Unless explicitly stated otherwise, Banach spaces are real. For a complex space, its $(S)$-modulus is computed on the underlying real Banach space; thus the perturbations range over the full complex space.

The notation used in this paper is primarily based on \cite{AlbiacKalton:2006}. $S_X$ denotes the unit sphere of a normed space $X$. A Banach space $G$ is called \emph{Gurari\u{\i}} if it exhibits \emph{almost universal disposition for finite-dimensional normed spaces}, as follows:

\begin{itemize}
\item[] Given any finite-dimensional spaces $E\subseteq F$, an isometric embedding $T\colon E\to G$, and $\varepsilon > 0$, there exists an extension $T_\varepsilon$ of $T$ from $F$ into $G$ with $\|T_\varepsilon\| \|T_\varepsilon^{-1} \| < 1+ \varepsilon$.
\end{itemize}

The term pays homage to V.~I.~Gurari\u{\i}, who demonstrated the existence of separable Gurari\u{\i} spaces in \cite{Gurarii:1966}. Furthermore, there exists a unique separable Gurari\u{\i} space up to isometry. Garbuli\'nska-W\k{e}grzyn and Kubi\'s constructed Gurari\u{\i} spaces with arbitrary density in \cite{GarbKubis}. For constructions of spaces of universal and almost universal disposition, and their relationship with separable injectivity, see \cite{Avilesetal:2011,Avilesetal:2017}.

Recently, Tursi \cite{Tursi:2020} introduced a lattice-theoretic analogue of Gurari\u{\i} spaces. Specifically, Tursi established the existence of a unique (up to isometry) separable Banach lattice, $\mathfrak{T}$, exhibiting the following property, which we designate as $(T)$:

\begin{itemize}
\item[] Given any two finitely generated Banach lattices $E\subseteq F$ with $E = \langle a_1, \ldots, a_n\rangle$, lattice-isometric embedding $J\colon E\to \mathfrak{T}$, and $\varepsilon > 0$, there exists a lattice-isometric embedding $\widehat{J}\colon F\to \mathfrak{T}$ such that $\max\limits_{i= 1, \ldots, n}\|Ja_i - \widehat{J}a_i\|<\varepsilon$.
\end{itemize}

The lattice $\mathfrak{T}$ was constructed as the Fra\"{\i}ss\'e limit of the category of finitely generated Banach lattices, similarly to how the Gurari\u{\i} space can be identified as the Fra\"{\i}ss\'e limit of the category of finite-dimensional normed spaces (\cite[Theorem 3.3]{BenYaacov:2015}).

\begin{lemma}\label{lem:upper-level-selection}
Let $\nu$ be an atomless finite measure with $0<\nu(\Omega)=m<\infty$,
and let $q\in L_1(\nu;\mathbb R)$. For every $b\in[0,m]$ there is a
measurable set $A$ such that
\[
 \nu(A)=b,\qquad
 \int_Aq\,d\nu\geqslant\frac b m\int_\Omega q\,d\nu.
\]
\end{lemma}
\begin{proof}
The cases $b=0,m$ are immediate. For $0<b<m$, choose $c\in\mathbb R$
with $\nu(q>c)\leqslant b\leqslant\nu(q\geqslant c)$. Adjoin to
$\{q>c\}$ a subset of $\{q=c\}$ of the required measure. Atomlessness
permits this choice. The values on the resulting set are at least those
on its complement, so its average is at least the average over $\Omega$.
This argument applies to signed integrable $q$ as well.
\end{proof}

\begin{lemma}\label{lem:ultrapower-modulus}
Let $X$ be a non-zero Banach space and let $X^{\mathscr U}$ be an
ultrapower of $X$. If $\xi=[(x_i)]$ and $\eta=[(y_i)]$ are unit vectors,
choose representatives with $x_i,y_i\in S_X$. Then
\[
 U_{X^{\mathscr U}}(\xi,\eta;a)
   =\lim_{\mathscr U}U_X(x_i,y_i;a)
 \qquad(a>0).
\]
Consequently, for $0<\varepsilon<d<2$,
\[
 U_X(d-\varepsilon;a)
 \leqslant U_{X^{\mathscr U}}(d;a)
 \leqslant U_X(d;a),
\]
and, for $0<d<d+\varepsilon<2$,
\[
 U_X(d;a)\leqslant U_{X^{\mathscr U}}(d+\varepsilon;a).
\]
In particular, $X$ has uniform property~{\rm(S)} if and only if
$X^{\mathscr U}$ has uniform property~{\rm(S)}.
\end{lemma}
\begin{proof}
Every $\zeta\in X^{\mathscr U}$ with $\|\zeta\|\leqslant a$ has
representatives $z_i$ with $\|z_i\|\leqslant a$: truncate an arbitrary
bounded representative radially at radius $a$. Therefore
\[
 \big|\|\xi+\zeta\|-\|\eta+\zeta\|\big|
 \leqslant\lim_{\mathscr U}U_X(x_i,y_i;a).
\]
Conversely, choose $z_i\in aB_X$ whose norm differences are within
an arbitrary fixed positive error of $U_X(x_i,y_i;a)$. Passing to
$\zeta=[(z_i)]$ and then letting the error tend to zero proves the
pairwise equality.

If $\|\xi-\eta\|\geqslant d$, then
$\|x_i-y_i\|\geqslant d-\varepsilon$ on a set in $\mathscr U$.
The pairwise equality gives the first inequality. The second follows
by using constant representatives. Similarly, if
$\|\xi-\eta\|\geqslant d+\varepsilon$, then $\|x_i-y_i\|\geqslant d$
on a set in $\mathscr U$, so
$U_X(d;a)\leqslant U_{X^{\mathscr U}}(\xi,\eta;a)$.
Take the infimum over such $\xi,\eta$ for the last displayed inequality.
The assertions about uniform property~{\rm(S)} follow from these bounds.
\end{proof}

\section{Proofs}
\begin{proof}[Proof of Theorem A]
For clause~\eqref{ThAi}, put $Y=L_p(\mu,X)$ and take $f,g\in S_Y$.
Write $D=\|f-g\|_p$ and $a_0=\min\{a,1\}$. The measure
\[
 d\nu=\|f(t)\|_X^p\,d\mu(t)
\]
is an atomless probability measure. Define the integrable real-valued
function
\[
 q(t)=\begin{cases}
 \displaystyle
 \frac{\|f(t)\|_X^p+\|f(t)-g(t)\|_X^p-\|g(t)\|_X^p}
 {\|f(t)\|_X^p},& f(t)\ne0,\\[1ex]
 0,&f(t)=0.
 \end{cases}
\]
The numerator vanishes where $f=0$, and
\[
 \int q\,d\nu
 =\|f\|_p^p+\|f-g\|_p^p-\|g\|_p^p=D^p.
\]
Choose a measurable set $A$ with
\[
 \nu(A)=a_0^p,\qquad \int_Aq\,d\nu\geqslant a_0^pD^p.
\]
This is Lemma~\ref{lem:upper-level-selection}; for $a_0=1$ take
$A=\Sigma$.

Set $z=-f\mathbf1_A$. Then $\|z\|_p=a_0\leqslant a$, and
\begin{align*}
 \|g+z\|_p^p-\|f+z\|_p^p
 &=\int_A\bigl(\|g-f\|_X^p-\|g\|_X^p+\|f\|_X^p\bigr)\,d\mu\\
 &=\int_Aq\,d\nu\geqslant a_0^pD^p.
\end{align*}
Since both perturbed norms are at most $1+a_0$, the mean value theorem
gives
\[
 \|g+z\|_p-\|f+z\|_p
 \geqslant\frac{a_0^pD^p}{p(1+a_0)^{p-1}}.
\]
Consequently,
\begin{equation}\label{eq:Bochner-lower}
 U_{L_p(\mu,X)}^{(p)}(d;a)\geqslant a_0^pd^p,
 \qquad
 U_{L_p(\mu,X)}(d;a)
 \geqslant\frac{a_0^pd^p}{p(1+a_0)^{p-1}}>0.
\end{equation}
This proves clause~\eqref{ThAi} for every $1\leqslant p<\infty$, without
any assumption on the separability or product structure of the measure
space.

For clause~\eqref{ThAii}, take $x,y\in S_G$, put $D=\|x-y\|$, and let
$E=\operatorname{span}\{x,y\}$. There is a linear isometry
$T:E\to L_1[0,1]$, by the isometric universality of real $L_1$ for
spaces of dimension at most two
\cite{Herz:1963,Lindenstrauss:1964}. For $\varepsilon>0$ put
$b=a/(1+\varepsilon)$. Theorem~\ref{L1} provides $w\in L_1[0,1]$ such
that $\|w\|_1\leqslant b$ and
\[
 \big|\|Tx+w\|_1-\|Ty+w\|_1\big|
 \geqslant\left(\frac{4b}{2+D}\wedge1\right)D-\varepsilon.
\]
By almost universal disposition, the inverse of $T$ extends from $T(E)$
to an embedding
\[
 S:\operatorname{span}\{Tx,Ty,w\}\longrightarrow G
\]
with $\|S\|\|S^{-1}\|<1+\varepsilon$. Since $S$ is an isometry on
$T(E)$, both $\|S\|$ and $\|S^{-1}\|$ are at most $1+\varepsilon$.
In particular, $\|Sw\|\leqslant a$. The norms of $S(Tx+w)=x+Sw$ and
$S(Ty+w)=y+Sw$ differ from the respective original norms by at most
$\varepsilon(1+b)$. Letting $\varepsilon\downarrow0$ gives
\[
 U_G(x,y;a)\geqslant\left(\frac{4a}{2+D}\wedge1\right)D.
\]
The expression on the right is increasing in $D$, so it gives the
required lower bound for $U_G(d;a)$. Every Gurari\u{\i} space contains
$\ell_\infty^n$ almost isometrically for every $n$, and is therefore not
super-reflexive. Theorem~\ref{ThmC} gives the matching upper bound.

For clause~\eqref{ThAiii}, let $G$ have property~$(T)$ and take
$x,y\in S_G$. Put $E=\operatorname{latt}\{x,y\}$ and identify $E$ with
the constant functions in $L_1([0,1],E)$. For
$b=\min\{a,1\}$ let
\[
 w=-x\mathbf1_{[0,b]},\qquad
 F=\operatorname{latt}\{x,y,w\}\subseteq L_1([0,1],E).
\]
Here $E$ and $F$ are finitely generated Banach lattices, $E\subseteq F$,
and
\[
 \|w\|=b,\qquad
 \|y+w\|-\|x+w\|=b\|x-y\|.
\]
Given $\varepsilon>0$, property~$(T)$ supplies a lattice-isometric
embedding $S:F\to G$ such that
\[
 \|Sx-x\|<\varepsilon,\qquad \|Sy-y\|<\varepsilon.
\]
Thus $\|Sw\|=b\leqslant a$ and
\begin{align*}
 \big|\|y+Sw\|-\|x+Sw\|\big|
 &\geqslant\big|\|Sy+Sw\|-\|Sx+Sw\|\big|-2\varepsilon\\
 &=b\|x-y\|-2\varepsilon.
\end{align*}
Letting $\varepsilon\downarrow0$ proves
$U_G(d;a)\geqslant d\min\{a,1\}$ and in particular uniform
property~$(S)$. This uses only the approximate agreement on the specified
generators in the definition of~$(T)$; an exact extension is unnecessary.
\end{proof}

\begin{corollary}\label{cor:Bochner-T-bounds}
Let $Y=L_1(\mu,X)\ne\{0\}$, where $\mu$ is atomless, or let $Y$
be a Banach lattice with property~$(T)$. Then
\[
 d\min\{a,1\}\leqslant U_Y(d;a)
 \leqslant\left(\frac{4a}{2+d}\wedge1\right)d
 \qquad(0<d<2,\ a>0).
\]
\end{corollary}
\begin{proof}
The lower bounds were proved above. The space $L_1(\mu,X)$ contains an
isometric copy of a non-zero atomless scalar $L_1$-space and is not
super-reflexive, so its upper bound follows from Theorem~\ref{ThmC}.
Property~$(T)$, applied with $E=\{0\}$ and $F=\ell_\infty^n$, gives a
lattice-isometric copy of $\ell_\infty^n$ in $Y$ for every $n$. Hence
$Y$ is not super-reflexive, and Theorem~\ref{ThmC} applies here as well.
\end{proof}

\begin{proof}[Proof of Theorem~\ref{L1}]
Take real unit vectors $f,g\in L_1(\mu)$ and put $D=\|f-g\|_1>0$.
Define $q_f=0$ where $f=0$ and, elsewhere, set
\[
 q_f=\frac{|f|+|f-g|-|g|}{|f|}.
\]
Define $q_g$ symmetrically, and put
\[
 E_f=\{q_f>0\},\quad m_f=\int_{E_f}|f|\,d\mu,
 \qquad
 E_g=\{q_g>0\},\quad m_g=\int_{E_g}|g|\,d\mu.
\]
Both scores are non-negative and
\[
 \int q_f|f|\,d\mu=\int q_g|g|\,d\mu=D.
\]
For real scalars a pointwise calculation gives
\begin{align*}
 |f|\mathbf1_{E_f}+|g|\mathbf1_{E_g}
 &=\frac{|f|+|g|+|f-g|}{2}-|f|\mathbf1_{\{f=g\}}.
\end{align*}
Hence $0<m_f,m_g$ and
\[
 m_f+m_g\leqslant1+D/2.
\]
After interchanging $f,g$, assume $m_f\leqslant(2+D)/4$.
If $a\geqslant m_f$, take $A=E_f$. Otherwise apply
Lemma~\ref{lem:upper-level-selection} on $E_f$, with measure
$|f|\,d\mu$, to obtain $A\subseteq E_f$ such that
\[
 \int_A|f|\,d\mu=a,\qquad
 \int_Aq_f|f|\,d\mu\geqslant\frac{aD}{m_f}.
\]
In either case $h=-f\mathbf1_A$ has norm at most $a$, and
\begin{align*}
 \|g+h\|_1-\|f+h\|_1
 &=\int_A\bigl(|g-f|-|g|+|f|\bigr)\,d\mu\\
 &\geqslant\min\{aD/m_f,D\}\\
 &\geqslant\left(\frac{4a}{2+D}\wedge1\right)D.
\end{align*}
The last expression is increasing in $D$, so taking the infimum over
pairs with $D\geqslant d$ gives the required lower bound.
For the upper bound, a non-zero atomless real $L_1$-space is not
super-reflexive, and Theorem~\ref{ThmC} applies.
\end{proof}

\subsection{Complex spaces and von Neumann preduals}

For a complex Banach space $X$, write
\begin{equation}\label{eq:complex-modulus}
H_X(\varepsilon)=
\inf_{\substack{u\in S_X\\ \|v\|\geqslant\varepsilon}}
\left(\max_{|\zeta|=1}\|u+\zeta v\|-1\right)
\qquad(\varepsilon>0).
\end{equation}
This is the usual modulus of complex convexity. The space $X$ is
\emph{complex uniformly convex} if $H_X(\varepsilon)>0$ for every
$\varepsilon>0$; see \cite{DavisGarlingTomczak:1984}. The convention
$\|v\|=\varepsilon$ gives the same modulus, because a complex disc is the
convex hull of its boundary.

\begin{proposition}\label{prop:complex-convexity-S}
If $X$ is a complex uniformly convex Banach space, then
\begin{equation}\label{eq:complex-convexity-S-bound}
U_X(d;a)\geqslant\frac{1+a}{1+2a}\,
H_X\!\left(\frac{ad}{1+a}\right)>0
\qquad(0<d<2,\ a>0).
\end{equation}
In particular, the underlying real space of $X$ has uniform property~{\rm(S)}.
\end{proposition}

\begin{proof}
Fix $x,y\in S_X$ with $\|x-y\|\geqslant d$, put $h=x-y$ and
$r=a/(1+a)$. Since $\|rh\|\geqslant rd$, there is a complex scalar
$\zeta$ with $|\zeta|=1$ such that
\[
\|y+r\zeta h\|\geqslant1+H_X(rd).
\]
Set $t=r\zeta$, $\lambda=(1-t)^{-1}$ and $z=(\lambda-1)x$. Then
\[
\|z\|=\frac{|t|}{|1-t|}\leqslant\frac r{1-r}=a,
\qquad |\lambda|\geqslant\frac1{1+r}.
\]
The identities
\[
x+z=\lambda x,\qquad y+z=\lambda(y+th)
\]
give
\begin{align*}
\big|\|y+z\|-\|x+z\|\big|
&=|\lambda|\big(\|y+th\|-1\big)\\
&\geqslant\frac{H_X(rd)}{1+r}
=\frac{1+a}{1+2a}H_X\!\left(\frac{ad}{1+a}\right).
\end{align*}
Taking the infimum over $x,y$ proves the assertion.
\end{proof}

\begin{proof}[Proof of Theorem B$^{\prime}$]
By Haagerup's result proved in
\cite[Theorem~4.3]{DavisGarlingTomczak:1984}, the dual of every complex
$C^*$-algebra is complex uniformly convex. If $Y$ is a closed complex
linear subspace of $A^*$, then $H_Y(\varepsilon)\geqslant
H_{A^*}(\varepsilon)>0$. Proposition~\ref{prop:complex-convexity-S}
therefore applies to $A^*$ and to $Y$. Finally, the canonical inclusion
$\mathcal M_*\subseteq\mathcal M^*$ is a complex linear isometry, so
this also proves the assertion about von Neumann preduals.
\end{proof}

In particular, $\ell_1(\Gamma;\mathbb C)$ is the predual of
$\ell_\infty(\Gamma;\mathbb C)$ and has uniform property~$(S)$ in its
usual norm. Its underlying real space is the $\ell_1$-sum of Euclidean
planes $\ell_2^2(\mathbb R)$. The renorming in
Theorem~\ref{thm:l1renorming} concerns $\ell_1(\Gamma;\mathbb R)$.

\begin{proposition}\label{prop:complex-universal-upper}
Let $X$ be a non-zero complex normed space, regarded as a real normed space
when computing its $(S)$-modulus. Then
\[
 U_X(d;a)\leqslant d\min\{a,1\}
 \qquad(0<d\leqslant2,\ a>0).
\]
\end{proposition}
\begin{proof}
Choose $u\in S_X$ and complex scalars $\alpha,\beta$ with
$|\alpha|=|\beta|=1$ and $|\alpha-\beta|=d$. Put $x=\alpha u$ and
$y=\beta u$. For every $z\in X$, complex homogeneity and the reverse
triangle inequality give
\begin{align*}
 \big|\|x+z\|-\|y+z\|\big|
 &=\big|\|u+\overline\alpha z\|-\|u+\overline\beta z\|\big|\\
 &\leqslant|\alpha-\beta|\,\|z\|=d\|z\|.
\end{align*}
The same expression is at most $\|x-y\|=d$. Taking the supremum over
$\|z\|\leqslant a$, and then the infimum defining $U_X(d;a)$, proves the
claim.
\end{proof}

\begin{proposition}\label{prop:complex-L1-modulus}
Let $\mu$ be an atomless measure such that $L_1(\mu;\mathbb C)$ is
non-zero. Then
\[
 U_{L_1(\mu;\mathbb C)}(d;a)=d\min\{a,1\}
 \qquad(0<d\leqslant2,\ a>0).
\]
\end{proposition}
\begin{proof}
The upper bound follows from Proposition~\ref{prop:complex-universal-upper}.
For the lower bound take $x,y\in S_{L_1(\mu;\mathbb C)}$ and write
$D=\|x-y\|_1$. If $a\geqslant1$, the choice $z=-x$ gives
$|\|x+z\|_1-\|y+z\|_1|=D$. Suppose henceforth that $0<a<1$.

Define an atomless probability measure $\nu$ by
$d\nu=|x|\,d\mu$, and set
\[
 q=\begin{cases}
 \displaystyle\frac{|x-y|+|x|-|y|}{|x|},& |x|>0,\\[1ex]
 0,& |x|=0.
 \end{cases}
\]
The triangle inequality implies $0\leqslant q\leqslant2$, and
\[
 \int q\,d\nu
 =\|x-y\|_1+\|x\|_1-\|y\|_1=D.
\]
Lemma~\ref{lem:upper-level-selection} gives a measurable set $A$ with
\[
 \nu(A)=a,\qquad \int_Aq\,d\nu\geqslant aD.
\]

Put $z=-x\mathbf1_A$. Then $\|z\|_1=a$, and
\begin{align*}
 \|y+z\|_1-\|x+z\|_1
 &=\int_A\bigl(|y-x|-|y|+|x|\bigr)\,d\mu\\
 &=\int_Aq\,d\nu\geqslant aD.
\end{align*}
Thus $U_{L_1}(x,y;a)\geqslant a\|x-y\|_1$ for $0<a<1$.
Taking the infimum over pairs with $\|x-y\|_1\geqslant d$ completes
the proof.
\end{proof}

\subsection{Proof of Theorem~\ref{ThmC}}

\begin{proof}
As $X$ is not super-reflexive, for every $\varepsilon>0$ the space $X$ contains a subspace
that is $(1+\varepsilon)$-isomorphic to $\ell_\infty^2$; this follows from James' uniformly
non-square theorem \cite{James:1964}. Passing to an ultrapower $X^{\mathscr U}$ with respect
to a countably incomplete ultrafilter, these almost isometric copies yield an \emph{isometric}
copy of $\ell_\infty^2$ inside $X^{\mathscr U}$. In particular, the unit sphere of
$X^{\mathscr U}$ contains a line segment of length $2$.

We first perform the geometric calculation in the whole ultrapower
$X^{\mathscr U}$. It contains unit vectors $v,w$ with $\|v-w\|=2$
and $\overline{vw}\subseteq S_{X^{\mathscr U}}$. In the calculation below,
write $Y=X^{\mathscr U}$.

Fix $d\in(0,2]$ and $a>0$. Let $x$ be the point of $\overline{vw}$ at distance $1-d/2$ from $v$
and let $y$ be the point of $\overline{vw}$ at distance $1-d/2$ from $w$. Then $x,y\in S_Y$ and
$\|x-y\|=d$. Let $z\in Y$ with $\|z\|\le a$ and, by symmetry, assume $\|x+z\|\le \|y+z\|$.
Write
\[
  y=t x + (1-t)w,
  \qquad
  t:=\frac{2-d}{2+d},
\]
so that
\[
  1-t=\frac{2d}{2+d}=\frac{d}{1+d/2}= \frac{2(d/2)}{1+d/2}.
\]
Then $y+z=t(x+z)+(1-t)(w+z)$ and hence
\begin{align*}
  \|y+z\|
  &\le t\|x+z\|+(1-t)\|w+z\|
   \\&\le t\|x+z\|+(1-t)(1+a)\\
  &=\|x+z\|+(1-t)\bigl(1+a-\|x+z\|\bigr)
   \\&\le \|x+z\|+(1-t)\bigl(1+a-(1-a)\bigr)\\
  &=\|x+z\|+2a(1-t)
   =\|x+z\|+\frac{4a}{2+d}\,d,
\end{align*}
where we used $\|w+z\|\le \|w\|+\|z\|\le 1+a$ and $\|x+z\|\ge \|x\|-\|z\|\ge 1-a$.
Therefore,
\[
  \bigl|\ \|x+z\|-\|y+z\|\ \bigr|\le \frac{4a}{2+d}\,d.
\]
Since also $\bigl|\|x+z\|-\|y+z\|\bigr|\le \|x-y\|=d$ for every $z$, we obtain
\[
  U_Y(x,y;a)\le \Big(\frac{4a}{2+d}\wedge 1\Big)d.
\]
For $0<d<2$, apply the preceding calculation inside $X^{\mathscr U}$
to a pair $\xi,\eta$ at distance $s$, where $d<s<2$.
Choose unit representatives $x_i,y_i$. On a set in $\mathscr U$,
$\|x_i-y_i\|>d$, and Lemma~\ref{lem:ultrapower-modulus} yields
\[
 U_X(d;a)\leqslant
 U_{X^{\mathscr U}}(\xi,\eta;a)
 \leqslant\left(\frac{4a}{2+s}\wedge1\right)s.
\]
Let $s\downarrow d$ and use continuity of the right-hand side.
For $d=2$, any antipodal pair $x,-x$ satisfies
$U_X(x,-x;a)=2\min\{a,1\}$, which gives the required endpoint estimate.
\end{proof}

\subsection{Proof of Theorem~\ref{thm:l1renorming}}

We begin by recording elementary bounds that will be used repeatedly:
\begin{equation}\label{eq:S-basic}
  \|u\|_1 \le \|u\|_S
  = \big(\|u\|_1^2+\|u\|_2^2\big)^{1/2}
  \le \sqrt{2}\,\|u\|_1
  \qquad(u\in \ell_1(\Gamma)).
\end{equation}

We shall also use, without further mention, the following \emph{symmetries}.
For $\varepsilon=(\varepsilon_\gamma)_{\gamma\in\Gamma}\in\{-1,1\}^\Gamma$, the coordinatewise
sign change $T_\varepsilon u := (\varepsilon_\gamma u(\gamma))_{\gamma\in\Gamma}$ is a linear
onto isometry of $(\ell_1(\Gamma),\|\cdot\|_1)$, $(\ell_1(\Gamma),\|\cdot\|_2)$ and
$(\ell_1(\Gamma),\|\cdot\|_S)$, hence it preserves the $(S)$-modulus.
Moreover, $U(x,y;a)=U(y,x;a)$ by definition (note that swapping $x$ and $y$ does \emph{not}
preserve the individual values of $\|x\|_1$, $\|x\|_2$, etc.; we only use the symmetry of $U$).

\begin{lemma}\label{lem:squared-transfer}
Let $x,y\in \ell_1(\Gamma)$ satisfy $\|x\|_S=\|y\|_S=1$, and let $z\in \ell_1(\Gamma)$ with
$\|z\|_S\le \theta$. Assume
\[
  \|y+z\|_2\ge \|y\|_2 \qquad\text{and}\qquad \|x+z\|_2\le \|x\|_2.
\]
Then
\[
  \|y+z\|_S^2-\|x+z\|_S^2
  \ \ge\ (\|y+z\|_1^2-\|y\|_1^2)+(\|x\|_1^2-\|x+z\|_1^2),
\]
Whenever the right-hand side of this inequality is non-negative, it follows that
\[
  \|y+z\|_S-\|x+z\|_S
  \ \ge\
  \frac{(\|y+z\|_1^2-\|y\|_1^2)+(\|x\|_1^2-\|x+z\|_1^2)}{2(1+\theta)}.
\]
\end{lemma}

\begin{proof}
By the $\ell_2$-assumptions,
\[
  \|y+z\|_2^2-\|x+z\|_2^2 \ \ge\ \|y\|_2^2-\|x\|_2^2.
\]
Since $\|x\|_S=\|y\|_S=1$, we have $\|y\|_2^2-\|x\|_2^2=\|x\|_1^2-\|y\|_1^2$, hence
\begin{align*}
\|y+z\|_S^2-\|x+z\|_S^2
&=(\|y+z\|_1^2-\|x+z\|_1^2)+(\|y+z\|_2^2-\|x+z\|_2^2)\\
&\ge (\|y+z\|_1^2-\|x+z\|_1^2)+(\|x\|_1^2-\|y\|_1^2)\\
&=(\|y+z\|_1^2-\|y\|_1^2)+(\|x\|_1^2-\|x+z\|_1^2).
\end{align*}
For the second estimate, assume this non-negativity. If both perturbed
norms vanish, the conclusion is immediate. Otherwise, write
\[
\|y+z\|_S-\|x+z\|_S
=\frac{\|y+z\|_S^2-\|x+z\|_S^2}{\|y+z\|_S+\|x+z\|_S},
\]
and use the non-negative lower bound for the numerator together with
$\|y+z\|_S+\|x+z\|_S\leqslant2(1+\theta)$.
\end{proof}

\begin{proof}[Proof of Theorem~\ref{thm:l1renorming}]
Fix $0<d<2$ and $a>0$, and put
\[
\alpha=\frac14\min\{a,d,1\}.
\]
Let $x,y\in S_{(\ell_1(\Gamma),\|\cdot\|_S)}$ satisfy
$\|x-y\|_S\geqslant d$. We shall construct $z$ with
$\|z\|_S\leqslant\alpha$ and
\begin{equation}\label{eq:target-gap-D}
 \big|\|x+z\|_S-\|y+z\|_S\big|
 \geqslant 2^{-38}\alpha^9.
\end{equation}

Set
\[
 B=\{\gamma:x(\gamma)y(\gamma)<0\},\qquad
 B^c=\Gamma\setminus B.
\]
By a coordinatewise sign change, we may assume that $x\geqslant0$
everywhere, $y\geqslant0$ on $B^c$, and $y\leqslant0$ on $B$.
At coordinates where $x=0$, choose the sign so that $y\geqslant0$.

\smallskip\noindent
\emph{Case 1: $\|(x-y)\mathbf1_B\|_1>\alpha/2$.}
On $B$ the absolute values of $x$ and $y$ add, so one of
$\|x\mathbf1_B\|_1$, $\|y\mathbf1_B\|_1$ exceeds $\alpha/4$.
By interchanging $x,y$ and then changing signs on $B$ if necessary, we
may assume that $\|x\mathbf1_B\|_1>\alpha/4$ while retaining the sign
conventions above. Choose $\lambda\in(0,1]$ so that
\[
 z=-\lambda x\mathbf1_B,\qquad \|z\|_S=m:=\alpha/4.
\]
Then $|x+z|\leqslant|x|$ and $|y+z|\geqslant|y|$ coordinatewise.
Writing $b=\|z\|_1$, we have
\[
 \|x+z\|_1=\|x\|_1-b,\qquad
 \|y+z\|_1=\|y\|_1+b.
\]
Since $\|x\|_1,\|y\|_1\geqslant1/\sqrt2$ and
$b\geqslant m/\sqrt2$, Lemma~\ref{lem:squared-transfer} gives
\[
 \|y+z\|_S-\|x+z\|_S
 \geqslant\frac{b(\|x\|_1+\|y\|_1)}{1+m}
 \geqslant\frac{m}{1+m}
 \geqslant\frac\alpha8.
\]
Thus \eqref{eq:target-gap-D} holds in this case.

\smallskip\noindent
\emph{Case 2: $\|(x-y)\mathbf1_B\|_1\leqslant\alpha/2$.}
Since $d\geqslant4\alpha$, \eqref{eq:S-basic} yields
\begin{align*}
 \|(x-y)\mathbf1_{B^c}\|_1
 &\geqslant\frac{d}{\sqrt2}-\frac\alpha2\\
 &\geqslant\left(2\sqrt2-\frac12\right)\alpha
 >2\alpha.
\end{align*}
Consequently, after interchanging $x,y$ and changing signs on $B$ if
necessary, the set
\[
 A=\{\gamma\in B^c:y(\gamma)<x(\gamma)\}
\]
satisfies
\begin{equation}\label{eq:one-sided-l1-gap}
 u:=(x-y)\mathbf1_A\geqslant0,\qquad \|u\|_1\geqslant\alpha.
\end{equation}
We keep this ordered pair and this set $A$ for the rest of the proof.
In particular, $0\leqslant u\leqslant x$ and $\|u\|_1\leqslant1$.

\smallskip\noindent
\emph{Subcase 2a: $\|u\|_2\leqslant(\alpha/16)^2$.}
Put
\[
 A_1=\{\gamma\in A:u(\gamma)\leqslant(\alpha/16)x(\gamma)\},
 \qquad A_2=A\setminus A_1.
\]
If $\|u\mathbf1_{A_2}\|_1\leqslant\alpha/16$, then
\[
 1\geqslant\|x\|_1
 \geqslant\frac{16}{\alpha}\|u\mathbf1_{A_1}\|_1
 \geqslant\frac{16}{\alpha}\left(\alpha-\frac\alpha{16}\right)=15,
\]
a contradiction. Thus $\|u\mathbf1_{A_2}\|_1>\alpha/16$.
Every coordinate of $u$ is at most $(\alpha/16)^2$, so a finite subset
$A_3\subset A_2$ can be chosen with
\begin{equation}\label{eq:A3-window-D}
 (\alpha/16)^2<\|u\mathbf1_{A_3}\|_1\leqslant\alpha^2/128.
\end{equation}
Indeed, add coordinates until their sum first exceeds $(\alpha/16)^2$.
The definition of $A_2$ gives
\begin{equation}\label{eq:xA3-D}
 \|x\mathbf1_{A_3}\|_1
 \leqslant\frac{16}{\alpha}\|u\mathbf1_{A_3}\|_1
 \leqslant\alpha/8.
\end{equation}
Split $A_3=A_4\cup A_5$, where
\[
 A_4=\{\gamma\in A_3:x(\gamma)\leqslant2y(\gamma)\},
 \qquad A_5=A_3\setminus A_4.
\]
At least one of $\|u\mathbf1_{A_4}\|_1$,
$\|u\mathbf1_{A_5}\|_1$ exceeds $\alpha^2/512$.

If $\|u\mathbf1_{A_4}\|_1>\alpha^2/512$, put
$z=-2y\mathbf1_{A_4}$. By \eqref{eq:xA3-D},
\[
 \|z\|_S\leqslant\sqrt2\|z\|_1
 \leqslant2\sqrt2\|x\mathbf1_{A_3}\|_1
 \leqslant\alpha/2.
\]
On $A_4$, $y+z=-y$ and $x+z=x-2y$, with
$|x+z|=2y-x\leqslant x$. Thus $\|y+z\|_2=\|y\|_2$ and
$\|x+z\|_2\leqslant\|x\|_2$. Also
\[
 \|y+z\|_1=\|y\|_1,\qquad
 \|x\|_1-\|x+z\|_1=2\|u\mathbf1_{A_4}\|_1.
\]
Lemma~\ref{lem:squared-transfer} therefore yields
\begin{align*}
 \|y+z\|_S-\|x+z\|_S
 &\geqslant
 \frac{\|u\mathbf1_{A_4}\|_1}{\sqrt2(1+\alpha/2)}\\
 &\geqslant\frac{\alpha^2}{768\sqrt2}
 \geqslant2^{-11}\alpha^2.
\end{align*}

Otherwise $\|u\mathbf1_{A_5}\|_1>\alpha^2/512$; put
$z=-x\mathbf1_{A_5}$. Again \eqref{eq:xA3-D} implies
$\|z\|_S\leqslant\alpha/4$. On $A_5$, $x+z=0$ and
$|y+z|=x-y\geqslant y$. Hence $\|x+z\|_2\leqslant\|x\|_2$ and
$\|y+z\|_2\geqslant\|y\|_2$, whereas
\[
 (\|y+z\|_1-\|y\|_1)+(\|x\|_1-\|x+z\|_1)
 =2\|u\mathbf1_{A_5}\|_1.
\]
Both summands on the left are non-negative. Since
$\|x\|_1,\|y\|_1\geqslant1/\sqrt2$, Lemma~\ref{lem:squared-transfer}
gives the same bound $2^{-11}\alpha^2$ as in the preceding paragraph.
This proves \eqref{eq:target-gap-D} in Subcase~2a.

\smallskip\noindent
\emph{Subcase 2b: $\|u\|_2>(\alpha/16)^2$.}
Put
\[
 \Delta=\|x\|_1-\|y\|_1,\qquad
 Q=(\alpha/16)^4,\qquad \eta=Q/4.
\]
For every $w\geqslant0$ supported in $B^c$ and every $s\geqslant0$,
expanding the two squared norms gives the exact identity
\begin{equation}\label{eq:D-positive-perturbation}
 \|x+sw\|_S^2-\|y+sw\|_S^2
 =2s\sum_{\gamma\in\Gamma}
 w(\gamma)\bigl(\Delta+x(\gamma)-y(\gamma)\bigr).
\end{equation}
Here the constant terms cancel because $\|x\|_S=\|y\|_S=1$.
Whenever $0\ne w\leqslant |x|$ or $0\ne w\leqslant |y|$, choose
\[
 s=\min\{1,\alpha/\|w\|_S\},\qquad z=sw.
\]
Then $\|w\|_S\leqslant1$, so $s\geqslant\alpha$ and
$\|z\|_S\leqslant\alpha$. In all such cases, the sum of the two
perturbed norms is at most $2(1+\alpha)\leqslant4$.

If $\Delta\geqslant0$, take $w=u$. The sum in
\eqref{eq:D-positive-perturbation} is at least $\|u\|_2^2>Q$.
It follows that
\[
 \|x+z\|_S-\|y+z\|_S\geqslant\alpha Q/2.
\]
This implies \eqref{eq:target-gap-D}.

Suppose now that $\Delta<0$, and write $\delta=-\Delta>0$. Define
\[
 A^+=\{\gamma\in A:u(\gamma)\geqslant2\delta\},
 \qquad A^-=A\setminus A^+.
\]
If $\|u\mathbf1_{A^+}\|_2^2\geqslant\|u\|_2^2/2$, take
$w=u\mathbf1_{A^+}$. Then
\[
 \sum_\gamma w(\gamma)(\Delta+x(\gamma)-y(\gamma))
 \geqslant\frac12\|u\mathbf1_{A^+}\|_2^2
 >Q/4=\eta.
\]
Consequently the resulting norm difference is at least
$\alpha\eta/2$, which again implies \eqref{eq:target-gap-D}.

In the remaining case,
\[
 Q/2<\|u\mathbf1_{A^-}\|_2^2
 \leqslant2\delta\|u\mathbf1_{A^-}\|_1
 \leqslant2\delta,
\]
so $\delta>Q/4=\eta$. Put $C=B^c\setminus A$. Since $0\leqslant y\leqslant x$
on $A$, we have
\[
 \|y\mathbf1_B\|_1+\|y\mathbf1_C\|_1
 =\|y\|_1-\|y\mathbf1_A\|_1
 \geqslant\|x\|_1+\delta-\|x\mathbf1_A\|_1
 \geqslant\delta>\eta.
\]
If $\|y\mathbf1_B\|_1\geqslant\eta/2$, choose $\lambda\in(0,1]$
with
\[
 z=-\lambda y\mathbf1_B,\qquad \|z\|_S=\eta/2.
\]
This is possible because $\|y\mathbf1_B\|_S\geqslant\eta/2$;
also $\eta/2\leqslant\alpha$. The argument of Case~1, with $x,y$
interchanged, gives
\[
 \|x+z\|_S-\|y+z\|_S
 \geqslant\frac{\eta/2}{1+\eta/2}\geqslant\eta/4,
\]
which proves \eqref{eq:target-gap-D}.

Otherwise $\|y\mathbf1_C\|_1>\eta/2$. Take $w=y\mathbf1_C$ and
choose $s,z$ as above. On $C$ we have $x\leqslant y$ and hence
$\Delta+x-y\leqslant-\delta\leqslant-\eta$. Thus
\eqref{eq:D-positive-perturbation} gives
\[
 \|y+z\|_S^2-\|x+z\|_S^2
 \geqslant2s\eta\|y\mathbf1_C\|_1
 \geqslant\alpha\eta^2.
\]
Dividing by the sum of the perturbed norms now yields
\[
 \|y+z\|_S-\|x+z\|_S
 \geqslant\frac{\alpha\eta^2}{4}
 =\frac{\alpha^9}{64\cdot16^8}
 =2^{-38}\alpha^9.
\]
This proves \eqref{eq:target-gap-D} in the final case.

Since $\alpha\leqslant1$, every preceding lower bound is at least
$2^{-38}\alpha^9$. Therefore, taking the infimum over $x,y$, we obtain
\[
 U_{(\ell_1(\Gamma),\|\cdot\|_S)}(d;a)
 \geqslant2^{-56}\min\{a,d,1\}^9>0.
\]
The asserted uniform property~$(S)$ follows.
\end{proof}

\subsection{Proof of Corollary~\ref{cor:l1-bidual-renorming}}

\begin{proof}
For finite $\Gamma$, this follows from Theorem~\ref{thm:l1renorming}.
Suppose that $\kappa=|\Gamma|$ is infinite and put $K=\beta\Gamma$.
The atomic--atomless decomposition of regular measures gives
\[
 \ell_1(\Gamma;\mathbb R)^{**}
 =M(K)=\ell_1(K)\oplus_1 N(K),
\]
where $N(K)$ is the real Banach lattice of atomless regular signed
measures on $K$. We show that $N(K)$ absorbs the first summand.

First, $K$ admits a continuous surjection onto
\[
 Q\times\{0,1\}^{\mathbb N},\qquad Q=\{0,1\}^{2^\kappa}.
\]
Here is a direct verification. The set of pairs $(s,H)$, where
$s\subseteq\kappa$ is finite and $H\subseteq\mathcal P(s)$, has cardinal
$\kappa$. Each such pair determines a point $p_{s,H}$ of
$\{0,1\}^{\mathcal P(\kappa)}$ by
\[
 p_{s,H}(A)=1\quad\Longleftrightarrow\quad A\cap s\in H.
\]
These points are dense: given finitely many distinct subsets
$A_1,\ldots,A_n$ of $\kappa$, choose a finite $s$ so that their
intersections with $s$ are distinct, and then choose $H$ to realise any
prescribed values at those coordinates. Since
$|\mathcal P(\kappa)|+\aleph_0=2^\kappa$, the displayed product has a
dense subset of cardinal at most $\kappa$. A map from the discrete space
$\Gamma$ onto that subset extends continuously to $K$; its extension is
surjective by compactness. Write this surjection as $q$.
It also follows that
\[
 |K|=2^{2^\kappa}=|Q|:
\]
the lower bound follows from $q$, and the upper bound from identifying
points of $K$ with ultrafilters on $\Gamma$.

For each $t\in Q$, put
$K_t=q^{-1}(\{t\}\times\{0,1\}^{\mathbb N})$.
Lift the usual atomless Bernoulli probability on
$\{t\}\times\{0,1\}^{\mathbb N}$ to a regular probability measure
$\nu_t$ on $K_t$. Such a lift follows by extending the corresponding
positive functional from the pullback of the continuous functions on
the slice, using Hahn--Banach and the Riesz representation theorem.
Atoms of finite Radon measures on compact Hausdorff spaces are point
atoms. Thus the lift is atomless, since any atom would push forward to
a point atom of the Bernoulli measure. The sets $K_t$ are pairwise disjoint and closed, so the measures
$\nu_t$ are pairwise mutually singular.

By the Kuratowski--Zorn lemma, extend this family to a maximal family
$(\mu_i)_{i\in I}$ of pairwise mutually singular atomless probability
measures on $K$. In particular, $|I|\geqslant|Q|=|K|$. The standard
measure-band decomposition gives the lattice isometry
\[
 N(K)\cong\Big(\bigoplus_{i\in I}L_1(\mu_i;\mathbb R)\Big)_{\ell_1};
 \qquad\text{see \cite[p.~52, Theorem~9]{Lacey:1974}.}
\]
Indeed, take the absolutely continuous component with respect to each
$\mu_i$. Only countably many components of any fixed measure are
non-zero, their norms add, and a non-zero residual atomless measure
would contradict maximality.

In this representation the map
\[
 P(f_i)_{i\in I}
   =\Big(\Big(\int f_i\,d\mu_i\Big)\mathbf1\Big)_{i\in I}
\]
is a norm-one projection onto a subspace $C$ isometric to $\ell_1(I)$.
Set $Z=\ker P$. The map $v\mapsto(Pv,(I-P)v)$ is an isomorphism from
$N(K)$ onto $C\oplus_1 Z$. Since $|K|\leqslant|I|$ and $I$ is infinite,
\[
\begin{split}
 M(K)&=\ell_1(K)\oplus_1N(K)\\
 &\cong\ell_1(K)\oplus_1\ell_1(I)\oplus_1 Z
 \cong\ell_1(I)\oplus_1 Z\cong N(K).
\end{split}
\]
Finally, $N(K)$ is lattice-isometric to an atomless $L_1$-space (equivalently,
use the disjoint union of the probability spaces above), so it has uniform
property~$(S)$ by Theorem~\ref{L1}. Transporting its norm through the
isomorphism proves the assertion.
\end{proof}

\section{Borel complexity}\label{sec:borel-complexity}

We use the Wijsman quotient coding of separable Banach spaces from
\cite{Kania:2026}. The bound below concerns this specified Polish topology;
no claim of an optimal Borel rank is made.

\begin{proposition}\label{prop:uniformS-Pi4}
Let $\mathscr S$ be the Polish space of closed linear subspaces of
$\ell_1$, equipped with the Wijsman topology. Then
\[
 \{K\in\mathscr S:\ell_1/K\text{ has uniform property }(S)\}
\]
belongs to $\Pi^0_4$.
\end{proposition}
\begin{proof}
We use the quotient-coding framework of \cite{Kania:2026}. Let
$\mathscr S$ be the space of closed linear subspaces of $\ell_1$, equipped
with the Wijsman topology, and put $X_K=\ell_1/K$ for $K\in\mathscr S$.
The space $\mathscr S$ is Polish. Fix the countable dense rational vector
subspace $D=c_{00}(\mathbb Q)$ of $\ell_1$ and write
\[
 p_K(v)=\|v+K\|=\operatorname{dist}(v,K)\qquad(v\in D).
\]
Every map $K\mapsto p_K(v)$ is continuous. Every separable Banach space is
isometric to some $X_K$: send the unit vectors of $\ell_1$ to a sequence
dense in its open unit ball, and use successive approximation to obtain a
metric quotient. Thus this coding preserves the given norm, as required
for uniform property~$(S)$.

For positive integers $k,m,n$ and $v,w,z\in D$, let $O_{k,m,n,v,w,z}$
be the set of all $K\in\mathscr S$ for which at least one of the following
conditions holds:
\begin{align*}
 &|p_K(v)-1|>1/n,\\
 &|p_K(w)-1|>1/n,\\
 &p_K(v-w)<1/k,\\
 &p_K(z)<1/m\quad\text{and}\quad
       |p_K(v+z)-p_K(w+z)|>1/n.
\end{align*}
These sets are open. Put $\mathscr U=\{K\in\mathscr S:X_K\text{ has uniform property }(S)\}$. We claim that
\begin{equation}\label{eq:uniformS-Pi4}
 \mathscr U
 =\bigcap_{k,m=1}^{\infty}\ \bigcup_{n=1}^{\infty}
     \ \bigcap_{v,w\in D}\ \bigcup_{z\in D}O_{k,m,n,v,w,z}.
\end{equation}
Here $\mathbb Q$ may be replaced by $\mathbb Q(i)$ for complex spaces.
This will give the claimed $\Pi^0_4$ bound, since the innermost union is
open, the next intersection is $G_\delta$, the following union is
$\Sigma^0_3$, and the outer intersection is $\Pi^0_4$.

The zero quotient satisfies both sides of \eqref{eq:uniformS-Pi4} if the unit-sphere conditions are read as vacuous. We may therefore assume $X_K\ne\{0\}$. Suppose first that $X_K$ has uniform property~$(S)$, and fix $k,m$.
Set $r=1/k$, $b=1/m$, and
$\gamma=U_{X_K}(r/2;b/2)>0$. Choose $n\geqslant2$ such that
\[
 \varepsilon:=1/n<\min\{r/4,\gamma/8\}.
\]
Take $v,w\in D$. If one of the first three conditions defining
$O_{k,m,n,v,w,z}$ holds, there is nothing to prove. Otherwise the images
$x=v+K$ and $y=w+K$ satisfy
\[
 |\|x\|-1|\leqslant\varepsilon,\qquad
 |\|y\|-1|\leqslant\varepsilon,\qquad
 \|x-y\|\geqslant r.
\]
The unit vectors $x'=x/\|x\|$ and $y'=y/\|y\|$ satisfy
$\|x-x'\|,\|y-y'\|\leqslant\varepsilon$ and
$\|x'-y'\|\geqslant r-2\varepsilon\geqslant r/2$.
Hence there is $u\in X_K$, with $\|u\|\leqslant b/2$, such that
\[
 |\|x'+u\|-\|y'+u\||>\gamma/2.
\]
It follows that
$|\|x+u\|-\|y+u\||>\gamma/2-2\varepsilon>\varepsilon$.
Approximate $u$ sufficiently closely by $z+K$, with $z\in D$.
The strict inequalities then give $p_K(z)<b$ and
$|p_K(v+z)-p_K(w+z)|>\varepsilon$, proving the last condition.

Conversely, suppose that $K$ belongs to the right-hand side of
\eqref{eq:uniformS-Pi4}. Fix $d\in(0,2)$ and $a>0$, choose $k,m$ with
$1/k<d$ and $1/m<a$, and take the corresponding $n$.
Let $x,y\in S_{X_K}$ with $\|x-y\|\geqslant d$.
Choose sequences $(v_j)_{j=1}^{\infty}$ and $(w_j)_{j=1}^{\infty}$ in
$D$ such that $v_j+K\to x$ and $w_j+K\to y$.
For all sufficiently large $j$, none of the first three conditions holds.
Consequently there is $z_j\in D$ with $p_K(z_j)<1/m<a$ and
\[
 |p_K(v_j+z_j)-p_K(w_j+z_j)|>1/n.
\]
The reverse triangle inequality now gives
\[
 U_{X_K}(x,y;a)\geqslant
 \frac1n-\|x-(v_j+K)\|-\|y-(w_j+K)\|.
\]
Letting $j\to\infty$ yields $U_{X_K}(x,y;a)\geqslant1/n$.
Since $n$ does not depend on $x,y$, we obtain
$U_{X_K}(d;a)\geqslant1/n>0$, which proves the claim.
\end{proof}

\section{Open problems}\label{problems}
We conclude with questions that remain open after the preceding results.

\begin{enumerate}
    \item Does $c_0$ have equivalent renormings with uniform property $(S)$?
    \item Does every space have an equivalent renorming with uniform property $(S)$?
    \item\label{prob:upper-bound} Is there a universal constant $C>0$ such that
    $U_X(d;a)\leqslant Cad$ for every real Banach space $X$ of dimension
    at least two, every $0<d<2$, and every $a>0$?
    Proposition~\ref{prop:complex-universal-upper} gives the bound with
    $C=1$ for complex spaces, while Theorem~\ref{ThmC} gives $C=2$
    for non-super-reflexive real spaces.
    \item Characterise those pointed metric spaces whose Lipschitz-free Banach spaces have uniform property (S). For example, by Rademacher's theorem, $F([0,1])$ is isometric to $L_1[0,1]$ so it does have uniform property (S), unlike $F(\mathbb N)$ that is isometric to $\ell_1$.
    \item If $X$ and $Y$ are normed spaces and $p\in (1,\infty)$, is there a constant $C_p > 0$ so that the $\ell_p$-direct sum of $X$ and $Y$ satisfies
    \[
        U_{X\oplus_p Y}(d;a) \geqslant C_p\cdot \min\{U_X(d;a), U_Y(d;a), U_{\ell_p}(d;a)\},
    \]
    where $d\in (0, 2)$ and $a > 0$?
    \item\label{prob:borel} Determine the optimal Borel complexity of the
    class in Proposition~\ref{prop:uniformS-Pi4}. In particular, is it
    $\Pi^0_4$-complete in the Wijsman quotient coding?

    \item\label{prob:vn-modulus} Let $\mathcal M$ be a non-zero diffuse finite
    complex von Neumann algebra with a faithful normalised trace.
    Does its real Banach space of self-adjoint normal functionals,
    \[
       (\mathcal M_*)_{\rm sa}
       =\{\varphi\in\mathcal M_*:
       \varphi(b^*)=\overline{\varphi(b)}\text{ for all }b\in\mathcal M\},
    \]
    have uniform property~$(S)$? For the full complex predual, determine
    $U_{\mathcal M_*}(d;a)$. In particular, does
    \[
        U_{\mathcal M_*}(d;a)=d\min\{a,1\}
        \qquad(0<d<2,\ a>0)
    \]
    hold? Theorem~\ref{L1} and Proposition~\ref{prop:complex-L1-modulus}
    settle the respective questions when $\mathcal M=L_\infty[0,1]$.
\end{enumerate}

\subsection*{Acknowledgements} We are indebted to Professor Z. Lipecki (Wroc{\l}aw) for providing us with reference \cite{FastSwierczkowski}. \subsection*{Funding} Support from NCN Sonata-Bis 13, grant 2023/50/E/ST1/00067, is acknowledged with thanks.

\end{document}